\documentclass[
]{amsart}
\usepackage{algorithm, algorithmic}
\usepackage{multirow} 
\usepackage{amsmath}\usepackage{amsthm}\usepackage{amssymb}\usepackage{mathrsfs}\usepackage{amscd}\usepackage{graphicx}\usepackage{subfigure}\usepackage{amsfonts}\usepackage{amsxtra}\usepackage{color}
\usepackage{amscd}

\DeclareMathOperator{\rank}{rank}

\DeclareMathOperator{\trace}{tr}

\DeclareMathOperator{\diag}{diag}\DeclareMathOperator*{\spann}{span}\DeclareMathOperator{\supp}{supp}

\DeclareMathOperator{\Pol}{Pol}

\newcommand{\RR}{\mathbb{R}}

\theoremstyle{definition}
\newtheorem{definition}{Definition}[section]
\newtheorem{example}[definition]{Example}\newtheorem{remark}[definition]{Remark}
\theoremstyle{plain}\newtheorem{theorem}[definition]{Theorem}\newtheorem{corollary}[definition]{Corollary}\newtheorem{proposition}[definition]{Proposition}

\newcommand{\R}{\mathbb{R}}\newcommand{\N}{\mathbb{N}}

\begin{document}

\title[Moment fusion]{From low to high-dimensional moments \\ without magic}

\author[B.~G.~Bodmann]{Bernhard G.~Bodmann}
\address[B.~G.~Bodmann]{University of Houston, 
Department of Mathematics, 
Houston, TX 77204-3008
}
\email{bgb@math.uh.edu}

\author[M.~Ehler]{Martin Ehler}
\address[M.~Ehler]{University of Vienna,
Department of Mathematics,
Oskar-Morgenstern-Platz 1
A-1090 Vienna
}
\email{martin.ehler@univie.ac.at}

\author[M.~Gr\"af]{Manuel Gr\"af}
\address[M.~Gr\"af]{University of Vienna,
Department of Mathematics,
Oskar-Morgenstern-Platz 1
A-1090 Vienna
}
\email{manuel.graef@univie.ac.at}


%
\begin{abstract}
We aim to compute the first few moments of a high-dimensional random vector from the first few moments of a number of its low-dimensional projections. To this end, we identify algebraic conditions on the set of low-dimensional projectors that yield explicit reconstruction formulas. We also provide a computational framework, with which suitable projectors can be derived by solving an optimization problem. Finally, we show that randomized projections permit approximate recovery.  
\end{abstract}



\maketitle


\section{Introduction}
A central problem in dimension reduction, distributed sensing, and many statistical applications is the identification of properties of a high-dimensional random vector from knowledge of marginal distributions, i.e., the distributions of one or more lower-dimensional projections of the random vector. A simple example in statistics is the problem of computing its lowest moments. However, knowledge of some marginal distributions may not be sufficient to identify the first few high-dimensional moments. Here, we shall address the problem of designing low-dimensional projections of the random vector, so that its high-dimensional moments can be computed from the lower dimensional ones by an explicit formula.

We consider a random vector $X$ in $\R^d$, distributed according to some Borel probability measure. In practice, $X$ could be a random signal that is observed by 
distributed sensors, each measuring a certain piece of information. Inspired by \cite{Casazza:2008aa,Kutyniok:2009aa}, each sensor is modeled as a matrix $Q_j\in\R^{k\times d}$ with full rank $k< d$. 
Computing with $P_j:=Q_j^*(Q_jQ_j^*)^{-1}Q_j$ instead of $Q_j$, we can effectively turn our measurement matrices into orthogonal projectors $\{P_j\}_{j=1}^n\subset\mathcal{G}_{k,d}$, where $\mathcal{G}_{k,d}$ denotes the set of orthogonal projectors on $\R^d$ with rank $k$, i.e., $P_j$ is the orthogonal projector onto the row-space of $Q_j$. A variant of the Cram\'er-Wold Theorem says that two random vectors $X,Y\in\R^d$ are identically distributed if and only if, for \emph{all} $P\in\mathcal{G}_{k,d}$, the two random vectors $PX, PY$ are identically distributed, cf.~\cite{Renyi:1952mw}. For further related results on projected distributions, we refer to \cite{Heppes:1956xd,Gilbert:1955ts,Belisle:1997qc,J.Cuesta-Albertos:2007th}. Here, we do not wish to identify the distribution of $X$, but restrict us to
recover its first few moments. On the other hand, we want to achieve this by observing the moments of a number of
low-dimensional projections and combining the information in a process we call moment fusion. 



\subsection{Moment fusion}
Suppose $X$ is a random vector in $\R^d$ distributed according to some unknown Borel probability measure on $\R^d$. For a fixed integer $p>0$, our goal is to determine the low-order moments  
\begin{equation}\label{eq:1}
\mathbb{E}X^s, \quad s\in\N^d,\; |s|\leq p,
\end{equation}
from low-order moments of lower-dimensional projections. We use here multi-index notation $X^s=X_1^{s_1}\cdots X_d^{s_d}$ and $|s| = \sum_{j=1}^d s_j$. More specifically, we suppose that we have only access to the first $p$ moments of low-dimensional linear measurements, i.e., for certain matrices $\{Q_j\}_{j=1}^n\subset \R^{k\times d}$ with fixed rank $k<d$, we suppose that we know 
\begin{equation}\label{eq:2}
\mathbb{E} (Q_jX)^s , \quad s\in\N^k,\; |s|\leq p.
\end{equation}
From knowledge of $\{Q_j\}_{j=1}^n$ and the first $p$ moments of the dimension reduced random vectors $Q_jX$, $j=1,\ldots,n$, in \eqref{eq:2}, we aim to reconstruct the high-dimensional moments of $X$ in \eqref{eq:1}.

%

\subsection{Special examples}
Suppose that $x\in\R^d$ is a vector of unknowns. If $\{Q_j\}_{j=1}^n$ are chosen such that 
\begin{equation*}
\{(Q_jx)^{s} : j=1,\ldots,n,\; s\in\N^{k},\; |s|\leq p\}
\end{equation*}
spans the space of polynomials in $x$ of degree at most $p$, then \eqref{eq:1} can be computed from \eqref{eq:2}
by expressing each expected value $\mathbb E X^s$, $s\in\N^d$, $|s|\leq p$, in a suitable linear combination of the expected values
of $\{\mathbb E (Q_j X)^{s'}: j=1,\ldots,n,\; s'\in\N^{k},\; |s'|\leq p\ \}_{j=1}^n$. 
We provide an example for $k=1$:
\begin{example}\label{ex:1}
\begin{enumerate}
\item
If $p=1$, then we can simply choose $Q_j:=e_j^*$, $j=1,\ldots,d$, where $\{e_j\}_{j=1}^d$ is the standard orthonormal basis for $\R^d$. So, reconstruction is possible with $d$ projectors. 
\item For $p=2$, the $\{Q_j\}_{j=1}^d$ together with $Q^+_{i,j}:=e^*_i+e^*_j$, $i<j$ allow to reconstruct \eqref{eq:1} from \eqref{eq:2} with $\binom{d+1}{2}$ many low-dimensional measurements.
\item 
If $p=3$, one can check that $\{Q_j\}_{j=1}^n$ with $Q^+_{i,j}$, $Q^-_{i,j}:=e^*_i-e^*_j$, $i<j$, and $Q_{i,j,k}:=e^*_i+e^*_j+e^*_k$, $i<j<k$, allow reconstruction, so that we use $\binom{d+2}{3}$ many linear measurements. 
\item For $p=4$, we can choose $\{Q_j\}_{j=1}^n$ with $Q^+_{i,j}$, $Q^-_{i,j}$, $\tilde{Q}^+_{i,j}:=e^*_i+2e^*_j$, $i<j$, and $Q_{i,j,k}$, $Q^{-}_{i,j,k}:=e^*_i-e^*_j+e^*_k$, $\tilde{Q}^{-}_{i,j,k}:=e^*_i+e^*_j-e^*_k$, $i<j<k$, and $Q_{i,j,k,\ell}:=e^*_i+e^*_j+e^*_k+e^*_\ell$, $i<j<k<\ell$, allow reconstruction, so that we use $\binom{d+3}{4}$ many  linear measurements. 
\end{enumerate}
\end{example}
Note that the number of  linear measurements in Example \ref{ex:1} is exactly the dimension of the homogeneous polynomials of degree $p$ in $d$ variables. Similar examples can be derived for more general situations, and the following example deals with $k=2$:
\begin{example}\label{ex:2}
\begin{enumerate}
\item
If $p=1$, then the choice $\binom{e_1^*}{e_2^*}$, $\binom{e_3^*}{e_4^*},\ldots$, up to $\binom{e_{d-1}^*}{e_{d}^*}$, for $d$ even or up to $\binom{e_{d}^*}{e_{1}^*}$, for $d$ odd, enables us to reconstruct the high-dimensional mean from the lower-dimensional means.
\item For $p=2$, moment reconstruction works with the $\binom{d}{2}$ projectors $\binom{e_i^*}{e_j^*}$, for $i<j$. 
\end{enumerate}
\end{example}

\subsection{Outline and contribution of this paper}
The present paper is dedicated to go beyond the explicit Examples \ref{ex:1} and \ref{ex:2}, and instead, provide a general strategy for moment reconstruction.
Our main contribution is the identification of conditions on the projectors that yield explicit reconstruction formulas. Moreover, such conditions are compatible with numerical schemes, meaning that suitable projectors can be constructed explicitly by minimizing a certain potential function as discussed in Sections \ref{section:approx} and \ref{sec:Polt}. We also discuss randomized constructions. Our approach stems from applied harmonic analysis and relates to the concept of so-called Grassmannian cubatures, see, for instance, \cite{Bachoc:2002aa,Bachoc:2010aa}. 

The remainder of the paper is organized as follows: The condition on the projections and the associated reconstruction formula are formulated in Section~\ref{sec:main res} for random vectors $X$ on the unit sphere $\mathbb{S}^{d-1}$. In Section \ref{sec:3} we deal with $X\in\R^d$ and either limit us to up to third moments or use rank one projections. Sections \ref{section:approx} and \ref{sec:Polt} are dedicated to the construction of suitable projectors based on numerical optimization and on a randomization strategy. 
The results in Section \ref{sec:error} imply that suitably randomized constructions can provide approximate moment recovery, with an error bound that holds with overwhelming probability.

\section{Main reconstruction results for $X\in\mathbb{S}^{d-1}$}\label{sec:main res}

In this section, we focus on random vectors $X$ with values in the unit sphere $\mathbb{S}^{d-1}$ of $\mathbb R^d$.
We will rely on some results on cubatures for polynomial spaces on the Grassmannian manifold, see \cite{Bachoc:2004fk,Bachoc:2002aa,Bachoc:2010aa,Ehler:2014zl,Ehler:2014qc}.

\subsection{General moment reconstruction}
We shall make use of the trace inner product $\langle M_1 ,M_2\rangle:=\trace(M_1M_2)$, for $M_1,M_2\in\mathscr{H}_d:=\{M\in\R^{d\times d}:M^\top=M\}$. The Grassmann space  
\begin{equation*}
\mathcal{G}_{k,d}:=\{P\in\mathscr{H}_d:P^2=P,\; \trace(P)=k\}
\end{equation*}
is the set of rank-$k$ orthogonal projections on $\R^{d}$. Note that $\{Q_j\}_{j=1}^n\subset \R^{k\times d}$ with all matrices having full rank $k<d$, and $P_j:=Q_j^*(Q_jQ_j^*)^{-1}Q_j$, for $j=1,\ldots,n$, yields $\{P_j\}_{j=1}^n\subset\mathcal{G}_{k,d}$. In place of $\{Q_j\}_{j=1}^n$ we shall find conditions on $\{P_j\}_{j=1}^n$ that enable moment reconstruction, i.e., conditions on the respective row-spaces of $\{Q_j\}_{j=1}^n$.  

The orthogonal group $\mathcal{O}(d)$ acts transitively on $\mathcal{G}_{k,d}$ by conjugation $P\mapsto UPU^*$, for $P\in\mathcal{G}_{k,d}$ and $U\in \mathcal{O}(d)$. Thus, there is an orthogonally invariant probability measure $\sigma_{k,d}$ on $\mathcal{G}_{k,d}$, which is induced by the Haar measure on $\mathcal{O}(d)$. This measure leads to the \emph{trace moments} 
\begin{equation*}
\mu_{k,d}(M_1,\ldots,M_t):=\int_{\mathcal{G}_{k,d}}\langle P,M_1\rangle\cdots \langle P,M_t\rangle d\sigma_{k,d}(P),\qquad \{M_i\}_{i=1}^t\subset\mathscr{H}_d,
\end{equation*}
which were introduced in \cite{Ehler:2014zl,Ehler:2014qc}. In the present section, we can restrict ourselves to 
\begin{equation*}
\mu^t_{k,d}(M):=\mu_{k,d}(M,\ldots,M),
\end{equation*}
where $M$ occurs $t$ times, and in Section \ref{sec:3} we shall make use of the more general case. In the following result, we use the notation $E_{x,y}:=\frac{1}{2}(xy^*+yx^*)$, for $x,y\in\R^d$.
\begin{theorem}\label{th:fundamental}
For $\alpha\in\N^d$ with $|\alpha|=t$, there are $y^\alpha_{s,i}\in \mathbb{S}^{d-1}$ and coefficients $f^\alpha_{s,i}\in\R$, such that 
\begin{equation*}
x^\alpha = \sum_{s=1}^t \sum_{i=1}^m f^\alpha_{s,i}\mu^s_{k,d}(E_{x,y^\alpha_{s,i}}),\quad\text{for all } x\in \mathbb{S}^{d-1}.
\end{equation*}
\end{theorem}
\begin{proof}
Note that \cite[Lemma 7.1]{Ehler:2014qc} yields
\begin{align}
x_{i_1}\cdots x_{i_t} &= \frac{1}{t!}\sum_{ J\subset\{1,\ldots,t\}}(-1)^{t+\# J}\big(\sum_{j\in J} x_{i_j}  \big)^t,\label{eq:lemma}
\intertext{and $\sum_{j\in J} x_{i_j}=\sum_{j\in J} \langle x,e_{i_j}\rangle$ leads to}
x_{i_1}\cdots x_{i_t} &=\frac{1}{t!}\sum_{ J\subset\{1,\ldots,t\}}(-1)^{t+\# J}\big(\langle x,\sum_{j\in J} e_{i_j}\rangle   \big)^t\label{eq:lemma combinatorial basis}.
\end{align}
Thus, it is sufficient to check that each $\langle x,y\rangle^t$, for $x,y\in \mathbb{S}^{d-1}$, can be written as a linear combination of terms $\mu^s_{k,d}(E_{x,y})$, $s=1,\ldots,t$.  

We shall prove the statement by an induction over $t$. The case $t=1$ is covered by
\begin{equation*}
\langle x,y\rangle=\frac{d}{k}\mu_{k,d}^1(E_{x,y}),
\end{equation*}
see, for instance, \cite{Bachoc:2010aa}. 

To consider general $t$, we need some preparations. A partition of $t$ is an integer vector $\pi=(\pi_1,\dots,\pi_t)$ whose entries are ordered by $\pi_1\geq \ldots\geq \pi_t\geq 0$ and sum up to $t=\sum_{i=1}^t\pi_i$. We denote the number of nonzero entries by $l(\pi)$, and the set of partitions $\pi$ of $t$ with $l(\pi)\leq d$ is denoted by $\mathscr{P}_{t,d}$.  

According to invariant theory, cf.~\cite[Theorem 7.1]{Procesi:1976ul}, the expansion 
\begin{equation}\label{eq:symmi}
\mu^t_{k,d}(M) = \frac{1}{q_{t,d}}\sum_{\pi\in\mathscr{P}_{t,d}} \alpha_\pi \trace(M^{\pi_1})\cdots\trace(M^{\pi_{l(\pi)}})
\end{equation}
holds with suitable real-valued coefficients $q_{t,d}$ and $\alpha_\pi$. 
For $x,y\in \mathbb{S}^{d-1}$, we observe that $\trace(E_{x,y}^s)$ is a polynomial of degree $s$ in $\langle x,y\rangle$ with leading coefficient $(\frac{1}{2})^{s-1}$. The latter yields, for $x,y\in \mathbb{S}^{d-1}$, that $\mu^t_{k,d}(E_{x,y})$ is a polynomial in $\langle x,y\rangle$ of degree $t$, i.e., 
\begin{align*}
\mu^t_{k,d}(E_{x,y})& = \frac{1}{q_{t,d}}\sum_{\pi\in\mathscr{P}_t} \alpha_\pi \big((\frac{1}{2})^{t-l(\pi)}\langle x,y\rangle^t+\sum_{s=1}^{t-1}c_{\pi,s}\langle x,y\rangle^s  \big)\\
& = \frac{\langle x,y\rangle^t}{2^t q_{t,d}} (\sum_{\pi\in\mathscr{P}_t} \alpha_\pi 2^{l(\pi)}) +\sum_{\pi\in\mathscr{P}_t} \sum_{s=1}^{t-1}c_{\pi,s}\langle x,y\rangle^s.
\end{align*}
One can check that its leading coefficient does not vanish. Indeed, denoting by $D_2$ a diagonal rank-2 projection matrix, we observe
\begin{equation*}
\frac{1}{2^t q_{t,d}} (\sum_{\pi\in\mathscr{P}_t} \alpha_\pi 2^{l(\pi)}) = \mu^t_{k,d}(D_2)=\int_{\mathcal{G}_{k,d}}\langle P,D_2\rangle^td\sigma_{k,d}(P)
\end{equation*}
and the right-hand-side is positive since the function $\langle \cdot ,D_2\rangle^t\geq 0$ does not vanish entirely on $\mathcal{G}_{k,d}$.   
Therefore, we can isolate $\langle x,y\rangle^t$ and write it as a linear combination of $\mu^t_{k,d}(E_{x,y})$ and terms $\langle x,y\rangle^s$, for $s=1,\ldots,t-1$. By induction, each term $\langle x,y\rangle^s$, $s=1,\ldots,t-1$ can be written as a linear combination of terms $\mu^s_{k,d}(E_{x,y})$, $s=1,\ldots,t-1$, which concludes the proof. 
\end{proof}
Theorem \ref{th:fundamental} represents monomials by linear combinations of $\mu^s_{k,d}(E_{x,y^\alpha_{s,i}})$, for some $y^\alpha_{s,i}\in \mathbb{S}^{d-1}$. Next, we shall aim to replace the latter with projected monomials. Let us define a function space on $\mathcal{G}_{k,d}$ by
\begin{align}
\Pol^\ell_t(\mathcal{G}_{k,d})&:=\spann\{\langle M,\cdot\rangle^s\big|_{\mathcal G_{k,d}} : M\in\mathscr{H}_d ,\;\rank(M)\leq \ell,\; s\leq t \}\label{def:2}
\end{align}
and introduce the concept of cubatures on $\mathcal{G}_{k,d}$.
\begin{definition}\label{def:cuba}
For $\{P_j\}_{j=1}^n\subset\mathcal{G}_{k,d}$ and $\{\omega_j\}_{j=1}^n\subset\R$, we say that $\{(P_j,\omega_j)\}_{j=1}^n$ is a \emph{cubature for $\Pol^{\ell}_{t}(\mathcal{G}_{k,d})$} if 
 \begin{equation*}
\sum_{j=1}^n\omega_jf(P_j) =\int_{\mathcal{G}_{k,d}} f(P)d\sigma_{k,d}(P),\quad\text{for all } f\in \Pol^{\ell}_{t}(\mathcal{G}_{k,d}).
\end{equation*}
\end{definition}
Note that the construction of cubatures for $\Pol^{\ell}_{t}(\mathcal{G}_{k,d})$ and the properties of the function space $\Pol^\ell_t(\mathcal{G}_{k,d})$ are discussed in more detail in the Sections \ref{section:approx} and \ref{sec:Polt}, respectively. We can now formulate our first result on moment reconstruction, which is a direct consequence of Theorem \ref{th:fundamental}.
\begin{corollary}\label{cor:3}
If $\{(P_j,\omega_j)\}_{j=1}^n$ is a cubature for $\Pol^2_{t}(\mathcal{G}_{k,d})$, then, for $\alpha\in\N^d$ with $|\alpha|=t$, there are coefficients $a^\alpha_\beta\in\R$, such that, for any random vector $X\in \mathbb{S}^{d-1}$,
\begin{equation}\label{eq:fundament}
\mathbb{E}X^\alpha =\sum_{|\beta|\leq t}  a^\alpha_\beta \sum_{j=1}^n \omega_j\mathbb{E} (P_jX)^\beta.
\end{equation}
\end{corollary}
\begin{proof}
We observe that $\langle P_j x,y\rangle = \langle P_j,E_{x,y}\rangle$ holds, for all $x,y\in\R^d$. 
According to Theorem \ref{th:fundamental} and since $\langle \cdot,E_{x,y}\rangle^t|_{\mathcal{G}_{k,d}}\in \Pol^{2}_{t}(\mathcal{G}_{k,d})$, 
the cubature property yields
\begin{equation*}
\mathbb{E}X^\alpha = \sum_{s=1}^t \sum_{i=1}^m f^\alpha_{s,i}\sum_{j=1}^n\omega_j \mathbb{E}\langle P_jX,y^\alpha_{s,i}\rangle^s
\end{equation*}
and the assertion follows by observing that the terms $\mathbb{E}\langle P_jX,y^\alpha_{s,i}\rangle^s$ are linear combinations of moments of order $s$ of $P_jX$. 
\end{proof}
Since we are originally given the moments of $Q_jX$, we must still express $\mathbb{E}(P_jX)^\beta$, where $P_j=Q_j^*(Q_jQ_j^*)^{-1}Q_j$ and $\beta\in\N^d$, $|\beta|\leq t$, by means of moments of $Q_jX$. If we suppress the index $j$, we obtain 
the multilinear relation
\begin{equation*}
\mathbb{E}(PX)_{i_1}\cdots(PX)_{i_t} =\sum_{j_1,\ldots,j_t=1}^k \mathbb{E}\big((QX)_{j_1}\cdots(QX)_{j_t}\big) z_{j_1,i_1}\cdots z_{j_t,i_t},
\end{equation*}
where $z_{i,k}=(Q^*(QQ^*)^{-1})_{i,k}$. Thus, the moments of $Q_jX$ enable us to apply \eqref{eq:fundament}.

%
%

\subsection{Explicit moment reconstruction}
This section is dedicated to explicitly compute the expansion in Corollary \ref{cor:3} for $t=1,2,3$ by using a very particular class of polynomial functions. Indeed, zonal polynomials, cf.~\cite{James:1964mz,James:1974aa,Gross:1987bf,Chikuse:2003aa,Muirhead:1982fk}, are special multivariate homogeneous polynomials on $\mathscr{H}_d$. These polynomials $C_\pi$ are indexed by all partitions $\pi$ of $\N$ and are invariant under orthogonal conjugation. According to \cite{James:1964mz}, see also \cite{Ehler:2014zl,Ehler:2014qc}, we obtain 
\begin{equation}\label{eq:123}
\mu^t_{k,d}(M) = \sum_{\pi\in\mathscr{P}_{t,d}}\frac{C_\pi(M)C_\pi(D_k)}{C_\pi(D_k)},\quad\text{for all } M\in\mathscr{H}_d,
\end{equation}
where $D_k$ is a diagonal matrix in $\R^{d\times d}$ with $k$ ones and zeros elsewhere. Knowledge of the zonal polynomials enabled us in \cite{Ehler:2014qc} to compute the trace moments for arbitrarily large $t$ and with explicit expressions for $t=1,2,3$:
\begin{theorem}[\cite{Ehler:2014qc}]\label{th:all contained}
For all $d\geq 3$ and $k<d$, we have
\begin{align*}
\mu^1_{k,d}(M) &=\frac{1}{q_{1,d}}\alpha_{(1)}\trace(M),\\
\mu^2_{k,d}(M) &=\frac{1}{q_{2,d}}\big(\alpha_{(1,1)}\trace^2(M)+\alpha_{(2)}\trace(M^2)\big),\\
\mu^3_{k,d}(M) &=\frac{1}{q_{3,d}}\big(\alpha_{(1,1,1)}\trace^3(M)+\alpha_{(2,1)}\trace(M)\trace(M^2)+\alpha_{(3)}\trace(M^3)\big),
\end{align*}
holds for all $M\in\mathscr{H}_d$, where $q_{1,d}=d$, $\alpha_{(1)}=k$, and 
\begin{align*}
q_{2,d}&=(d-1)d(d+2),\\
\alpha_{(1,1)} &=(d+1)k^2-2k\\
\alpha_{(2)} &= 2k(d-k),\\
q_{3,d}&= (d-2)(d-1)d(d+2)(d+4),\\
\alpha_{(1,1,1)}& =(d^2+3d-2)k^3-6(d+2)k^2+16k,\\
\alpha_{(2,1)}& =-6(d+2)k^3  +  6(d^2+2d+4) k^2-24dk ,\\
\alpha_{(3)}& = 16k^3-24dk^2+8d^2k.
\end{align*}
\end{theorem}
For $d=2$ and $k=1$ in Theorem \ref{th:all contained}, the constant $q_{3,d}$ would be zero, but so are $\alpha_{(1,1,1)}$, $\alpha_{(2,1)}$, and $\alpha_{(3)}$. The identity for $\mu^3_{1,2}(M)$ still holds with the modified coefficients
\begin{align*}
q_{3,2}= 48,\quad \alpha_{(1,1,1)}=1,\quad \alpha_{(2,1)}=6,\quad \alpha_{(3)}=8.
\end{align*}

Theorem \ref{th:all contained} and the proof of Theorem \ref{th:fundamental} lead to the following explicit moment recovery formulas.
\begin{corollary}\label{main rec theorem}
Let $X\in \mathbb{S}^{d-1}$ be a random vector with $d \ge 3$. 
\begin{itemize}
\item[(i)] If $\{(P_j,\omega_j)\}_{j=1}^n$ is a cubature for $\Pol^2_{1}(\mathcal{G}_{k,d})$, then, for $i=1,\ldots,d$,  
\begin{equation}\label{eq:1a}
\mathbb{E}X_i=A_1\sum_{j=1}^n\omega_j \mathbb{E}(P_jX)_i, \qquad\text{where $A_1=\frac{d}{k}$.}
\end{equation}
\item[(ii)] If $\{(P_j,\omega_j)\}_{j=1}^n$ is a cubature for $\Pol^2_{2}(\mathcal{G}_{k,d})$, then \eqref{eq:1a} holds and, for $i,\ell=1,\ldots,d$,
\begin{equation}\label{eq:2b}
\mathbb{E}X_iX_\ell=A_2 \sum_{j=1}^n\omega_j  \mathbb{E}(P_jX)_i (P_jX)_\ell -B_2\frac{k}{d}\delta_{i,\ell},
\end{equation}
where 
\begin{equation*}
A_2=\frac{(d-1)d(d+2)}{k(dk+d-2)},\qquad B_2=\frac{(d-k)d}{kd(k+1)-2k}.
\end{equation*}
\item[(iii)] If $\{(P_j,\omega_j)\}_{j=1}^n$ is a cubature for $\Pol^{2}_{3}(\mathcal{G}_{k,d})$, then \eqref{eq:1a} and \eqref{eq:2b} hold and, for $i,\ell,m=1,\ldots,d$,
\begin{align*}
\mathbb{E}X_iX_\ell X_m&=A_3\sum_{j=1}^n\omega_j  \mathbb{E}(P_jX)_i(P_jX)_\ell (P_jX)_m\\
&-\frac{B_3}{3}\sum_{j=1}^n\omega_j  \big(\mathbb{E}(P_jX)_i\delta_{\ell,m}+\mathbb{E}(P_jX)_\ell\delta_{i,m}+\mathbb{E}(P_jX)_m\delta_{i,\ell}\big) \frac{k}{d},
\end{align*}
where 
\begin{align*}
A_3&= \frac{(d-2)(d-1)d(d+2)(d+4)}{k(d^2k^2+3d^2k+2d^2-6dk-12d-4k^2)}, \\
B_3&= \frac{3d^2(d^2k + 2d^2 - 5dk - 4d + 4k^2 + 2k)}{k^2(d^2k^2 + 3d^2k + 2d^2 + 3dk^2 - 9dk - 12d + 2k^2 - 6k + 16)} .
\end{align*}
\end{itemize}
\end{corollary}
\begin{remark}
Note that \eqref{eq:1a} is proved by monomial identities, so that it also holds when the expectation is eliminated on both sides. 
\end{remark}

\begin{proof}[Proof of Corollary~\ref{main rec theorem}]
For $x,y\in\R^d$, we obtain
\begin{align*}
\trace(E_{x,y}) & = \langle x,y\rangle,\\
\trace(E_{x,y}^2) & =\frac{1}{2}(\langle x,y\rangle^2+\|x\|^2\|y\|^2),\\
\trace(E_{x,y}^3) & =\frac{1}{4}(\langle x,y\rangle^3+3\langle x,y\rangle\|x\|^2\|y\|^2),
\end{align*}
so that Theorem \ref{th:all contained} implies, for all $d\geq 3$ and $x,y\in\R^d$,
\begin{align}
\mu^1_{k,d}(E_{x,y})  & = \frac{\alpha_{(1)}}{q_{1,d}}\langle x,y\rangle, \label{eq:a}\\
\mu^2_{k,d}(E_{x,y})  & = \frac{2\alpha_{(1,1)}+\alpha_{(2)}}{2q_{2,d}}\langle x,y\rangle^2+\frac{\alpha_{(2)}}{2q_{2,d}}\|x\|^2\|y\|^2,\label{eq:b}\\
\mu^3_{k,d}(E_{x,y})   & = \frac{4\alpha_{(1,1,1)}+2\alpha_{(2,1)}+\alpha_{(3)}}{4q_{3,d}}\langle x,y\rangle^3+\frac{2\alpha_{(2,1)}+\alpha_{(3)}}{4q_{3,d}}\langle x,y\rangle\|x\|^2\|y\|^2,\hspace{-1cm}\label{eq:c}
\end{align}
where the constants $q_{1,d},q_{2,d},q_{3,d}$ and $\alpha_{(1)},\alpha_{(2)},\alpha_{(3)},\alpha_{(1,1,1)},\alpha_{(2,1)}$ are specified in Theorem \ref{th:all contained}. 

Suppose now that $\{(P_j,\omega_j)\}_{j=1}^n$ is a cubature of $\Pol^{2}_{t}(\mathcal{G}_{k,d})$, so that we obtain
\begin{align*}
\mu_{k,d}^t(E_{x,y})& =\sum_{j=1}^n\omega_j \langle P_{j},E_{x,y}\rangle^t = \sum_{j=1}^n\omega_j \langle P_{j}x,y\rangle^t,
\end{align*}
and the left-hand-sides in \eqref{eq:a}, \eqref{eq:b}, and \eqref{eq:c} can be replaced with $\sum_{j=1}^n\omega_j \langle P_{j}x,y\rangle^t$, for $t=1,2,3$, respectively. In the following, we assume $x\in \mathbb{S}^{d-1}$ and $y \in \mathbb R^d$. Rearranging terms leads to the following formulas, respectively, and $A_1,A_2,A_3$, and $B_2,B_3$ are as in Corollary~\ref{main rec theorem}. If $\{(P_j,\omega_j)\}_{j=1}^n$ is a cubature for $\Pol^2_{1}(\mathcal{G}_{k,d})$, then  
\begin{equation}\label{eq:cub 1}
\langle x,y\rangle=A_1\sum_{j=1}^n\omega_j \langle P_{j}x,y\rangle.
\end{equation}
If $\{(P_j,\omega_j)\}_{j=1}^n$ is a cubature for $\Pol^2_{2}(\mathcal{G}_{k,d})$, then \eqref{eq:cub 1} holds and 
\begin{equation}\label{eq:cub 2}
\langle x,y\rangle^2=A_2 \sum_{j=1}^n\omega_j \langle P_{j}x,y\rangle^2-B_2\|y\|^2\frac{k}{d}.
\end{equation}
If $\{(P_j,\omega_j)\}_{j=1}^n$ is a cubature for $\Pol^{2}_{3}(\mathcal{G}_{k,d})$, then \eqref{eq:cub 1}, \eqref{eq:cub 2} hold, so that
\begin{equation}\label{eq:cub 3}
\langle x,y\rangle^3=
A_3\sum_{j=1}^n\omega_j \langle P_{j}x,y\rangle^3 - B_3\|y\|^2\sum_{j=1}^n\omega_j \langle P_{j}x,y\rangle\frac{k}{d}.
\end{equation}
As at the beginning of the proof of Theorem \ref{th:fundamental}, \cite[Lemma 7.1]{Ehler:2014qc} yields
\begin{align}
x_{i_1}\cdots x_{i_t} &=\frac{1}{t!}\sum_{J\subset\{1,\ldots,t\}}(-1)^{t+\# J}\big(\langle x,\sum_{j\in J} e_{i_j}\rangle   \big)^t\nonumber.
\end{align}
In order to compute the term $x_{i_1}\cdots x_{i_t}$, we can repeatedly apply \eqref{eq:cub 1}, \eqref{eq:cub 2}, \eqref{eq:cub 3}, respectively, with $y=\sum_{j\in J} e_{i_j}$. 
Such rearrangements yield the formulas and constants in Corollary~\ref{main rec theorem}. 
%
%

\end{proof}

 \begin{remark}\label{remark:1}
 The framework that we present in the present paper also allows the explicit computations of higher order moments beyond $t=1,2,3$. Indeed, if $\{(P_j,\omega_j)\}_{j=1}^n$ is a cubature for $\Pol^2_{t}(\mathcal{G}_{k,d})$, then we can compute all moments of order $t$ by using the zonal polynomials. However, we do not have one closed formula incorporating $t$ as a variable, but we need to compute the expressions for each $t$ separately. Note that the formulas in Theorem~\ref{main rec theorem} are merely based on identities between the corresponding polynomials in 
 the entries of a unit vector $x\in \mathbb R^d$.
 \end{remark}

\section{Moment fusion for $X\in\R^d$}\label{sec:3}
\subsection{The general case for up to cubic moments}
A homogeneity argument yields that \eqref{eq:1a} even holds for random $X\in\R^d$. Analogously, considering \eqref{eq:2b} as a monomial identity with $B_2\frac{k}{d}\delta_{i,\ell}=B_2\frac{k}{d}\|x\|^2\delta_{i,\ell}$, for $x\in \mathbb{S}^{d-1}$, a homogeneity argument yields that, for random $X\in\R^d$,
\begin{equation*}
\mathbb{E}X_iX_\ell=A_2 \sum_{j=1}^n\omega_j  \mathbb{E}(P_jX)_i (P_jX)_\ell -B_2\sum_{r=1}^d\sum_{j=1}^n\omega_j \mathbb{E}(P_jX)^2_r\delta_{i,\ell},
\end{equation*}
provided that $\{(P_j,\omega_j)\}_{j=1}^n$ is a cubature for $\Pol^2_{2}(\mathcal{G}_{k,d})$. 

In order to deal with $X\in\R^d$ for $t=3$ as well, we observe that the formulas in \eqref{eq:a}, \eqref{eq:b}, and \eqref{eq:c} hold in more generality, see \cite[Theorem 7.3]{Ehler:2014qc},
\begin{align*}
\mu_{k,d}(X_1) &=\frac{1}{q_{1,d}}\alpha_{(1)}\trace(X_1),\\
\mu_{k,d}(X_1,X_2)&= \frac{1}{q_{2,d}}\big(\alpha_{(1,1)}\trace(X_1)\trace(X_2)+\alpha_{(2)}\trace(X_1X_2)\big),\\
\mu_{k,d}(X_1,X_2,X_3)&=\frac{1}{q_{3,d}}\big( \alpha_{(1,1,1)}\trace(X_1)\trace(X_2)\trace(X_3)\\
&\quad\quad +\frac{\alpha_{(2,1)}}{3}(\trace(X_1)\trace(X_2X_3)+\trace(X_2)\trace(X_1X_3)+\trace(X_3)\trace(X_1X_2))\\
&\quad\quad +\alpha_{(3)} \trace(X_1X_2X_3) \big).
\end{align*}

For $x\in\R^d$ and $y\in \mathbb{S}^{d-1}$, using the above relation gives
$$
   \mu_{k,d}(E_{x,y},xx^*) = \frac{\alpha_{(1,1)} +\alpha_{(2)}}{q_{2,d}} \langle x, y \rangle \|x\|^2 = \frac{k(k+2)}{d(d+2)} \langle x, y \rangle \|x\|^2
$$
and combined with identity~(\ref{eq:c}), we obtain
\begin{equation}\label{eq:needed}
\langle x,y\rangle^3 = 
C_{3,d}^{(1)} \mu^3_{k,d}(E_{x,y}) - 
C_{3,d}^{(2)} \mu_{k,d}(E_{x,y},xx^*)
\end{equation}
with $C_{3,d}^{(1)} = \frac{4q_{3,d}}{4\alpha_{(1,1,1)}+2\alpha_{(2,1)}+\alpha_{(3)}}$ and 
$C_{3,d}^{(2)} = \frac{2\alpha_{(2,1)}+\alpha_{(3)}}{(4\alpha_{(1,1,1)}+2\alpha_{(2,1)}+\alpha_{(3)}) \frac{k(k+2)}{d(d+2)} }$.
Indeed, one can check that the term $4\alpha_{(1,1,1)}+2\alpha_{(2,1)}+\alpha_{(3)}$ is nonzero. If $\{(P_j,\omega_j)\}_{j=1}^n$ is a cubature for $\Pol_3^2(\mathcal{G}_{k,d})$, then we can apply
\begin{equation}\label{eq:one of two}
\mu^3_{k,d}(E_{x,y})  = \sum_{j=1}^n\omega_j \langle P_jx,y\rangle^3,
\end{equation}
because the mapping $P\mapsto \langle P,E_{x,y}\rangle^3$ is contained in $\Pol_3^2(\mathcal{G}_{k,d})$. In Proposition \ref{prop:3.1} we shall check that $P\mapsto \langle P, E_{x,y}\rangle \langle P, xx^*\rangle $ is also contained in $\Pol_3^2(\mathcal{G}_{k,d})$, so that also 
\begin{equation}\label{eq:two of two}
\mu_{k,d}(E_{x,y},xx^*)  = \sum_{j=1}^n\sum_{m=1}^d \omega_j \langle P_jx,y\rangle (P_jx)_m^2
\end{equation}
holds. The actual moments of order $3$ can now be computed from \eqref{eq:lemma combinatorial basis} by observing that $\langle P_jx,y\rangle$ yields a linear combination of the terms $(P_jx)_1,\ldots,(P_jx)_d$.

We now collect the resulting expressions for all third moments: 
\begin{corollary}
Let $X \in \mathbb R^d$ be a random vector with $d\ge 3$, let the constants $A_1$, $A_2$ and $B_2$ be as above, and let $i,h,l \in \{1, 2, \dots , d\}$ with $i \ne h \ne l \ne i$.
\begin{itemize}
\item[(i)] If $\{(P_j,\omega_j)\}_{j=1}^n$ is a cubature for $\Pol^2_{1}(\mathcal{G}_{k,d})$, then 
\begin{equation} \label{eq:1agen}
\mathbb{E}X_i=A_1\sum_{j=1}^n\omega_j \mathbb{E}(P_jX)_i.
\end{equation}
\item[(ii)] If $\{(P_j,\omega_j)\}_{j=1}^n$ is a cubature for $\Pol^2_{2}(\mathcal{G}_{k,d})$, then \eqref{eq:1agen} holds and 
\begin{align}
\mathbb{E}X^2_i &= A_2 \sum_{j=1}^n\omega_j  \mathbb{E}(P_jX)^2_i  -B_2\sum_{j=1}^n\sum_{m=1}^d\omega_j \mathbb{E}(P_jX)^2_m,\label{eq:2bgena} \\
\mathbb{E}X_iX_\ell &= A_2 \sum_{j=1}^n\omega_j  \mathbb{E}(P_jX)_i (P_jX)_\ell.\label{eq:2bgenb}
\end{align}
\item[(iii)] If $\{(P_j,\omega_j)\}_{j=1}^n$ is a cubature for $\Pol^{2}_{3}(\mathcal{G}_{k,d})$, then \eqref{eq:1agen}, \eqref{eq:2bgena}, \eqref{eq:2bgenb} hold and
\begin{align}
  \mathbb E X_i^3  & = C_{3,d}^{(1)} \sum_{j=1}^n \omega_j \mathbb E (P_j X)_i^3
                              - C_{3,d}^{(2)} \sum_{j=1}^n \sum_{m=1}^d \omega_j \mathbb E (P_jX)_i (P_jX)_m^2 , \label{eq:3 one}\\
\mathbb E X_i^2 X_h = & \, C_{3,d}^{(1)} \sum_{j=1}^n \omega_j \mathbb E (P_j X)_i^2 (P_jX_h)  
  - \frac 1 3 C_{3,d}^{(2)} \sum_{j=1}^n \sum_{m=1}^d \omega_j \mathbb E 
    (P_j X)_h (P_j X)_m^2 ,\label{eq:3 two}\\
  \mathbb E X_i X_h X_\ell & = C_{3,d}^{(1)} \sum_{j=1}^n \omega_j  \mathbb E (P_j X)_i (P_j X)_h (P_j X)_\ell .\label{eq:3 three}
\end{align}
\end{itemize}
\end{corollary}

\begin{proof}
The first and second moments have already been discussed prior to the statement of the corollary. For the third moments,
the expression for $\mathbb E X_i^3$ results from choosing $y=e_i$ in \eqref{eq:needed}, \eqref{eq:one of two}, and \eqref{eq:two of two}, which yields \eqref{eq:3 one}. 

Next, we address \eqref{eq:3 two}. The choices $y_+ = \frac{1}{\sqrt{2}} (e_i + e_h)$ and $y_- = \frac{1}{\sqrt{2}}(e_i + e_h)$  yield
 \begin{align*}
    x_i^2 x_h 
    & = \frac{\sqrt 2}{ 3 } [ \langle x , y_+ \rangle^3 - \langle x, y_-\rangle^3] - \frac{1}{3} \langle x, e_h \rangle ^3 .
 \end{align*}
Applying  \eqref{eq:needed}, \eqref{eq:one of two}, and \eqref{eq:two of two} lead to
 \begin{align*}
   \mathbb E X_i^2 X_h = & \, C_{3,d}^{(1)}  \sum_{j=1}^n \omega_j  ( (P_jX)_i^2 (P_j X)_h + \frac 1 3 (P_j X)_h^3 )\\
      & - \frac 2 3 C_{3,d}^{(2)} \sum_{j=1}^n \sum_{m=1}^d \omega_j (P_jX)_h (P_jX)_m^2 - \frac 1 3 \mathbb E X_h^3
   \end{align*}
   and inserting the expression for $\mathbb E X_h^3$ then reduces this to  \eqref{eq:3 two}. 
   
Finally, for $\mathbb E X_i X_h X_\ell$, we observe 
\begin{equation*}
x_i x_h x_\ell = \frac{1}{24} ( (x_i + x_h +x_\ell)^3 + (x_i - x_h -x_\ell)^3 - (x_i + x_h -x_\ell)^3 - (x_i - x_h +x_\ell)^3 ).
\end{equation*}
By using $y_{+++} = \frac{1}{\sqrt 3} (e_i + e_h + e_\ell)$,
$y_{+--} = \frac{1}{\sqrt 3} (e_i - e_h - e_\ell)$, $y_{--+} = - \frac{1}{\sqrt 3} (e_{i}+e_{j}-e_\ell)$, 
and $y_{-+-} = -\frac{1}{\sqrt 3} (e_i - e_h + e_\ell)$, we obtain
\begin{align*}
 \mathbb E X_i X_h X_\ell & = \frac{ \sqrt 3}{8} \mathbb E[ \langle X, y_{+++}\rangle ^3 
 + \langle X, y_{+--}\rangle ^3 + \langle X, y_{--+}\rangle ^3 + \langle X, y_{-+-}\rangle ^3 ], 
 \end{align*}
 and a calculation using \eqref{eq:needed}, \eqref{eq:one of two}, and \eqref{eq:two of two} leads to \eqref{eq:3 three}.
\end{proof}

\subsection{All moments from projections onto one-dimensional subspaces}
In the previous section, we have outlined the recovery of the moments for $t=1,2,3$ and general $k$. To address all moments $t>3$, we now restrict us to $k=1$.

Let us denote the permutation group of $\{1,\ldots,t\}$ by $S_t$. We say a permutation $s\in S_t$ is associated to a partition $\pi$ and denote this by $s\sim \pi$ if there is a set of cycles $\{c_i\}_{i=1}^m$ such that $s=(c_1)\cdots(c_m)$ and the cardinality of $c_i$ equals $\pi_i$, for $i=1,\ldots,m$. Note that we also use the standard notation $c_i\in s$ for a cycle $c_i$ occurring in $s$. For $\{M_i\}_{i=1}^t\subset \mathscr{H}_d$, we use a cycle index $M_{c_i}:=M_{c_{i,1}}\cdots M_{c_{i,\ell_i}}$, where $c_i=(c_{i,1}\ldots c_{i,\ell_i})$. 

To clarify notation, we provide a simple example.
\begin{example}
For $t=4$, let the permutation $s$ be given by $\begin{pmatrix}
1&2&3&4\\
3&1&2&4
\end{pmatrix}
$, and suppose we have a set of matrices $\{M_i\}_{i=1}^4$. Then $s$ has the cyclic representation $(c_1)(c_2)=(1 3 2)(4)$ and is associated to the partition $(3,1,0,0)$. This implies $M_{c_1}=M_1M_3M_2$ and $M_{c_2}=M_4$.
\end{example}
Due to the orthogonal invariance of the Haar measure, the Grassmannian trace moments are invariant under the orthogonal group $\mathcal{O}_d$, i.e., 
\begin{equation*}
\mu_{k,d}(UM_1,\dots, UM_t) = \mu_{k,d}(M_1,\dots, M_t),\quad\text{for all } U\in\mathcal{O}_d.
\end{equation*}
A general result in invariant theory, cf.~\cite{Procesi:1976ul}, and the invariance of $\mu_{k,d}$ under permutations yield 
\begin{equation}\label{eq:poly sum in general}
\mu_{k,d}(M_1,\dots, M_t)= \sum_{\pi\in\mathscr{P}_t} \alpha_\pi \sum_{\substack{s\in S_t\\  s\sim \pi}} \prod_{c\in s} \trace(M_c), 
\end{equation}
where $\alpha_\pi\in\RR$, see also \eqref{eq:symmi} for $M_1=\ldots=M_t$. 

\begin{proposition}
\label{prop:positive alpha}
For $d \ge t$, $t \in \N_{0}$, and provided that $k=1$, the expansion \eqref{eq:poly sum in general} of the trace moments  
possesses only positive coefficients $\alpha_{\pi}$, $\pi \in \mathscr{P}_{t}$.
\end{proposition}
\begin{proof}
For any fixed permutation $\sigma:\{1,\dots,t\}\to\{1,\dots,t\}$, we consider the matrices
\[
M_{i} = 
\begin{cases}
e_{i}e_{\sigma(i)}^{*} + e_{\sigma(i)
}e_{i}^{*},& i\ne\sigma(i) \\  
e_{i}e_{i}^{*},& i = \sigma(i),
\end{cases}, \qquad i=1,\dots,t,
\]
where $\{e_{i}\}_{i=1}^d \in \R^{d}$ is the standard basis. 

Now, let $s \in S_{t}$ be another arbitrary permutation with some cycle $c \in s$. We denote the cardinality of $c$ by $l$. From
\[
\trace(M_{c}) = \trace( M_{c_{1}} \cdots M_{c_{l}} )
 = \sum_{k_{1},\dots,k_{l}=1}^{d} (M_{c_{1}})_{k_{1},k_{2}} (M_{c_{2}})_{k_{2},k_{3}} \cdots (M_{c_{l-1}})_{k_{l-1},k_{l}}(M_{c_{l}})_{k_{l},k_{1}} 
\]
we conclude by the definition of $M_{i}$ that the indices $k_{i}$ contribute to
the sum if and only if
$k_{i}, k_{i+1\, \mathrm{mod}\, l} \in \{ c_{i}, \sigma(c_{i}) \}$ for all
$i=1,\dots,l$. Equivalently, it must hold that
$k_{i} \in \{ c_{i}, \sigma(c_{i}) \} \cap \{ c_{i-1\,\mathrm{mod}\,l},
\sigma(c_{i-1\,\mathrm{mod}\,l}) \}$,
$i=1,\dots,l$. Since $c_{i} \ne c_{j}$ and $\sigma(c_{i}) \ne \sigma(c_{j})$ for
$i \ne j$, this can only happen if $c_{i} = \sigma(c_{i-1\,\mathrm{mod}\,l})$
for all $i=1,\dots,l$ or $c_{i-1\,\mathrm{mod}\,l} = \sigma(c_{i})$ for all
$i=1,\dots,l$. Hence, the trace of $M_{c}$ vanishes if and only if neither the
cycle $c$ nor its inverse $c^{-1}$ are contained in $\sigma$. More precisely,
\[
\trace(M_{c}) = 
\begin{cases}
1, & c \in \sigma \text{ or } c^{-1} \in \sigma,\\
0, & \text{else}.
\end{cases}
\]

Using these observations we obtain
\[
\begin{aligned}
\mu_{1,d}(M_{1},\dots,M_{t}) & =  \sum_{\pi\in\mathscr{P}_t} \alpha_\pi \sum_{\substack{s\in S_t\\  s\sim \pi}} \prod_{c\in s} \trace(M_c) \\
& =  \alpha_{\pi_\sigma} \#\{ s \in S_{t} : c \text{ or } c^{-1} \in \sigma, \forall c\in s\},
\end{aligned}
\]
where $\pi_\sigma$ is the partition associated to $\sigma$. Hence, $\pi_\sigma$ is a fraction of the trace moment
$\mu_{1,d}(M_{1},\dots,M_{t})$. It remains to verify that the latter is positive. 

Together with the definition of the trace moments $\mu_{1,d}(M_{1},\dots,M_{t})$ and those of $M_{i}$ we arrive at
\begin{align*}
\mu_{1,d}(M_{1},\dots,M_{t}) 
& = \int_{\mathcal O_{d}} \prod_{i=1}^{t} \langle O D_{1} O^{*}, M_{i} \rangle d O \\
& = \int_{\mathcal O_{d}} \prod_{i=1}^{t} \langle O e_{1} (O e_{1})^{*}, M_{i} \rangle d O \\
& = \int_{\mathcal O_{d}} \prod_{i=1}^{t} 2^{\#\{i,\sigma(i)\}} O_{1,i} O_{1,\sigma(i)} d O \\
& = \int_{\mathcal O_{d}} \Big(\prod_{i=1}^{t} 2^{\#\{i,\sigma(i)\}} O_{1,i}\Big) \Big(\prod_{i=1}^{t} O_{1,\sigma(i)} \Big)d O. 
\intertext{Since $\sigma$ is a permutation, we obtain}
\mu_{1,d}(M_{1},\dots,M_{t}) & =  \int_{\mathcal O_{d}} \Big(\prod_{i=1}^{t} 2^{\#\{i,\sigma(i)\}} O_{1,i}\Big) \Big(\prod_{i=1}^{t} O_{1,i} \Big)d O\\ 
&=\int_{\mathcal O_{d}} \prod_{i=1}^{t} 2^{\#\{i,\sigma(i)\}} (O_{1,i})^{2} d O > 0
\end{align*}
and the assertion follows.
\end{proof}

\begin{proposition}\label{th:ground breaking}
  For fixed $m,\ell \in \N_{0}$ and $d \in \N$ with $d \ge m + \ell$ there are
  coefficients $\{a_i\}_{i=0}^{\lfloor m/2\rfloor}\in\R$ such that, for all
  $x\in\R^d$, $y \in \mathbb{S}^{d-1}$, it holds
\begin{equation}
\label{eq:ground breaking}
\langle x,y\rangle^m\|x\|^{2\ell} =  \sum_{i=0}^{\lfloor m/2\rfloor} a_i \mu^{(m-2i,\ell+i)}_{1,d}(E_{x,y},xx^*).
\end{equation}
\end{proposition}

\begin{proof}
  Let us first note that by the identity \eqref{eq:poly sum in general} the
  trace moments $\mu^{(m,\ell)}_{1,d}(E_{x,y},xx^{*})$, $m,\ell \in \N_{0}$, can
  be written as polynomials in $\langle x, y \rangle$, $\|x\|^{2}$ and
  $\|y\|^{2}$. Hence, together with the homogeneity in $x$ and $y$ we infer the
  representation
  \begin{equation}
    \label{eq:mu Exy as scalar product}
    \begin{aligned}
      \mu^{(m,\ell)}_{1,d}(E_{x,y},xx^*) 
      = \sum_{i=0}^{\lfloor m/2 \rfloor}b_i^{(m,\ell)}\langle x,y\rangle^{m-2i}\|x\|^{2(i+\ell)},\qquad x \in \R^{d}, y \in \mathbb{S}^{d-1},
    \end{aligned}
  \end{equation}
  for some coefficients $b_{i}^{(m,\ell)} \in \R$. Moreover, we have 
  \begin{equation}
    \label{eq:b0_m_l}
    b_{0}^{(m,\ell)} > 0, \qquad  d \ge m + \ell, 
  \end{equation}
  which follows from Proposition \ref{prop:positive alpha} and the fact that the
  coefficients of $\langle x,y \rangle^{m} \| x \|^{2\ell}$ in any term of the
  form
  \[
  \prod_{i=1}^{r}\trace\big(\prod_{j=1}^{s_{i}} (E_{x,y})^{m_{i,j}}
  (xx^{*})^{\ell_{i,j}}\big), \quad \text{ with } \quad \sum_{i=1}^{r}\sum_{j=1}^{s_{i}}
  m_{i,j} = m, \quad \sum_{i=1}^{r} \sum_{j=1}^{s_{i}} \ell_{i,j} = \ell
  \]
  are positive.

  Now, the statement \eqref{th:ground breaking} will follow by induction over
  $m$. Therefore, let $m \ge 2$, $\ell \in \N_{0}$ with $d \ge m + \ell$ be
  given and assume that the statement \eqref{th:ground breaking} holds for all
  $m' = m-2i$, $\ell'=\ell + i$, $i=1,\dots,\lfloor m/2 \rfloor$. Using equation
  \eqref{eq:mu Exy as scalar product} we obtain
  \[
  \mu^{(m,\ell)}_{1,d}(E_{x,y},xx^*) 
  = b_{0}^{(m,\ell)} \langle x,y\rangle^{m} \|x\|^{2 \ell} +  \sum_{i=1}^{\lfloor m/2 \rfloor}b_i^{(m,\ell)}\langle x,y\rangle^{m-2i}\|x\|^{2(\ell+i)}.
  \]
  Since $d \ge m + \ell > m' + \ell' = m + \ell - i$,
  $i=1,\dots,\lfloor m/2 \rfloor$, we can expand the sum on the right hand side
  by the induction hypothesis into trace moments $\mu_{1,d}^{(m-2i,\ell+i)}(E_{x,y},xx^*)$,
  $i=1,\dots, \lfloor m/2 \rfloor$. Hence, using $b^{(m,\ell)}_{0} \ne 0$, see \eqref{eq:b0_m_l}, we can rearrange terms and arrive at the statement
  \eqref{eq:ground breaking}. It remains to show the induction base with the
  cases $m \in \{0,1\}$, $\ell \in \N_{0}$, $d \ge m + \ell$.

  For $m=0$, $\ell \in \N_{0}$ and $d \in \N$ we observe by orthogonal
  invariance
  \begin{equation*}
    \mu^{(0,\ell)}_{1,d}(E_{x,y},xx^*) = \mu^{\ell}_{1,d} (xx^{*})= \mu^{\ell}_{1,d}(D_{1})\|x\|^{2\ell} ,\qquad x\in\R^d.
  \end{equation*}
The term $\mu^{\ell}_{1,d}(D_{1})$ is positive and has been explicitely computed in \cite{Bachoc:2010aa}:
  \begin{equation*}
    \mu^{\ell}_{1,d}(D_1)=\frac{(1/2)_\ell}{(d/2)_\ell},\qquad (a)_\ell:=a(a+\ell)\cdots (a+\ell-1).
  \end{equation*}
  Hence, the assertion follows for $m=0$, $\ell \in \N_{0}$.

  For $m=1$, $\ell \in \N_{0}$ and $d \in \N$ we find by \eqref{eq:poly sum in
    general}  that
  \begin{equation*}
    \mu^{(1,\ell)}_{1,d}(E_{x,y},xx^*) = \langle x,y\rangle \|x\|^{2\ell} \sum_{\pi\in\mathscr{P}_{\ell+1}} \alpha_\pi \sum_{\substack{s\in S_{\ell + 1}\\  s\sim \pi}} 1, \qquad x\in\R^d,\quad y \in \mathbb{S}^{d-1}.
  \end{equation*} 
  Moreover, we can check that the coefficient of $\langle x,y\rangle \|x\|^{2\ell} $ is nonzero by observing
  \begin{align*}
    \sum_{\pi\in\mathscr{P}_{t}} \alpha_\pi \sum_{\substack{s\in S_{\ell+1}\\  s\sim \pi}} 1 = \mu^{\ell+1}_{1,d}(I_1)>0,
  \end{align*}
  which concludes the proof.
\end{proof}

\begin{corollary}\label{cor:zweiter Versuch}
If $\{(P_j,\omega_j)\}_{j=1}^n$ is a cubature for $\Pol^3_{t}(\mathcal{G}_{1,d})$, then, for $\alpha\in\N^d$ with $|\alpha|=t\leq d$, there are coefficients $a^\alpha_\beta\in\R$, such that, for any random vector $X\in \R^d$,
\begin{equation}\label{eq:fundament 2}
\mathbb{E}X^\alpha =\sum_{|\beta|= t}  a^\alpha_\beta \sum_{j=1}^n \omega_j\mathbb{E} (P_jX)^\beta.
\end{equation}
\end{corollary}
\begin{proof}
For any $i=0,\ldots,\lfloor t/2\rfloor$, the function $F:\mathcal{G}_{1,d}\rightarrow\R$ defined by  $P\mapsto \langle P,E_{x,y}\rangle^{t-2i}\langle P,xx^*\rangle^{2i}$ is contained in $\Pol^3_{t}(\mathcal{G}_{1,d})$, see part two of Proposition \ref{prop:to compare with} in the subsequent Section \ref{section:approx} for details.  Thus, the cubature property yields 
\begin{align*}
\mu^{(t-2i,i)}_{1,d}(E_{x,y_j},xx^*)&=\sum_{j=1}^n \omega_j \langle P_j,E_{x,y}\rangle^{t-2i} \langle P_j,xx^*\rangle^{i}\\
&= \sum_{j=1}^n \omega_j \langle P_jx,y\rangle^{t-2i} \|P_jx\|^{2i}.
\end{align*}
Note that $\langle P_jx,y\rangle^{t-2i}$ and $\|P_jx\|^{2i}$ are linear combinations of monomials in $P_jx$ of degree $t-2i$ and $2i$, respectively. Hence, their product is a linear combination of monomials in $P_jx$ of degree $t$. Applying \eqref{eq:lemma combinatorial basis} and invoking Proposition \ref{th:ground breaking} for $\ell=0$ concludes the proof. 
\end{proof}


\section{Constructing cubatures for $\Pol^{2}_{t}(\mathcal{G}_{k,d})$}\label{section:approx}
In this section we shall derive a general framework for the construction of cubatures for $\Pol^{2}_{t}(\mathcal{G}_{k,d})$ that are needed to apply our results in Theorem \ref{main rec theorem}. For general existence results of cubatures, we refer to \cite{Harpe:2005fk}, and explicit group theoretical constructions are provided in \cite{Calderbank:1997uq}. 
In the following we shall discuss random constructions as well as deterministic constructions based on the solution of an optimization problem.

\subsection{Random construction}\label{sec:random negative weights}
For $n,m\geq \dim({\Pol}^\ell_{t}(\mathcal{G}_{k,d}))$, it follows from classical arguments that there are 
$\{M_i\}_{i=1}^m\subset \mathscr{H}^\ell_d$ and $\{P_j\}_{j=1}^n\subset \mathcal{G}_{k,d}$ such that the matrix 
\begin{equation}\label{eq:matrix}
(\langle M_i,P_j\rangle^t)_{\substack{i=1\ldots,m\\ j=1,\ldots,n}}
\end{equation}
has rank $\dim({\Pol}^\ell_{t}(\mathcal{G}_{k,d}))$. We can now compute weights $\omega:=\{\omega\}_{j=1}^n$ by solving the linear system of equations
\begin{equation*}
\sum_{j=1}^n\langle M_i,P_j\rangle^t \omega_j =  \mu_{k,d}^t(M_i),\quad i=1,\ldots,m,
\end{equation*}
which yields a cubature $\{(P_j,\omega_j)\}_{j=1}^n$. Note that the weights $\{\omega_j\}_{j=1}^n$ are not necessarily nonnegative. 

We claim that $M_i$ and $P_j$ can be chosen in a random fashion. Indeed, we observe that both spaces, $\mathcal{G}_{k,d}$ and $\mathscr{H}^\ell_d$, can be parametrized analytically, so that there is $D>0$ and a surjective analytic mapping $F:\R^D\rightarrow (\mathscr{H}^\ell_d)^m\times (\mathcal{G}_{k,d})^n$. Let us assume $n=m=\dim({\Pol}^\ell_{t}(\mathcal{G}_{k,d}))$ for simplicity. Otherwise, we can extract a submatrix. We now define 
\begin{align*}
G:(\mathscr{H}^\ell_d)^n\times (\mathcal{G}_{k,d})^n & \rightarrow \R\\
\big((P_i)_{i=1}^n , (M_j)_{j=1}^n \big)& \mapsto \det\big((\langle P_i,M_j\rangle^t)_{i,j}\big).
\end{align*}
Since $F$ is surjective, the mapping $G\circ F:\R^D\rightarrow \R$ is not identically zero. 
Moreover, $G\circ F$ is analytic, so that $(G\circ F)^{-1}(\{0\})\subset\R^D$ has Lebesgue measure zero and, hence, is a zero set with respect to any continuous probability measure $\nu$ on $\R^D$. Thus, the parametrization $F$ enables a random choice in $(\mathscr{H}^\ell_d)^m\times (\mathcal{G}_{k,d})^n$, so that the matrix \eqref{eq:matrix} has rank $\dim({\Pol}^\ell_{t}(\mathcal{G}_{k,d}))$ with probability one with respect to $\nu$. In other words, $G^{-1}(\{0\})$ is a zero set with respect to the induced probability measure $\nu_F$ on $(\mathscr{H}^\ell_d)^m\times (\mathcal{G}_{k,d})^n$. Thus, \eqref{eq:matrix} has rank $\dim({\Pol}^\ell_{t}(\mathcal{G}_{k,d}))$ with probability one and weights $\{\omega_j\}_{j=1}^n$ can be computed.



Let us also verify that \eqref{eq:matrix} having rank $\dim({\Pol}^\ell_{t}(\mathcal{G}_{k,d}))$ is a generic property. Indeed, 
both spaces, $\mathcal{G}_{k,d}$ and $\mathscr{H}^\ell_d$, are real algebraic varieties that are irreducible, cf.~\cite{Bochnak:1998kx,Kutz:1974eu}, so that also $(\mathscr{H}^\ell_d)^m\times (\mathcal{G}_{k,d})^n$ is irreducible. Without loss of generality, we can restrict us to $n=m=\dim({\Pol}^\ell_{t}(\mathcal{G}_{k,d}))$ again. Note that $G$ 
is a polynomial map and, hence, is Zariski continuous. Therefore, the set $U:=\{u\in (\mathscr{H}^\ell_d)^m\times (\mathcal{G}_{k,d})^n : G(u)\neq 0\}$ is Zariski open. Classical arguments yield that $U$ cannot be empty, so that irreducibility yields that $U$ is Zariski dense. Thus, we have verified that, for $n,m\geq \dim({\Pol}^\ell_{t}(\mathcal{G}_{k,d}))$, there is a nonempty Zariski open and dense subset $U$ in $(\mathscr{H}^\ell_d)^m\times (\mathcal{G}_{k,d})^n$ such that the matrix 
\begin{equation*}
(\langle M_i,P_j\rangle^t)_{\substack{i=1\ldots,m\\ j=1,\ldots,n}},\quad \text{ where} \quad \big((M_i)_{i=1}^m,(P_j)_{j=1}^n\big)\in U,
\end{equation*}
has rank $\dim({\Pol}^\ell_{t}(\mathcal{G}_{k,d}))$.

\subsection{Deterministic construction}\label{sec:4.2}
Here, we present the design of cubatures as the solution of an optimization problem. As in \cite{Ehler:2014zl}, we shall apply the theory of reproducing kernel Hilbert spaces. We first define a measure $\nu_{\ell,d}$ on $\mathscr{H}^\ell_d:=\{M\in\mathscr{H}_d: \rank(M)\leq \ell\}$ by 
\begin{equation*}
\nu_{\ell,d}(\mathcal{A}):=\int_{\mathbb{S}^{\ell-1}}  \int_{\mathcal{O}_d} 1_{\mathcal{A}}(O^*\diag(\lambda_1,\ldots,\lambda_\ell,0,\ldots,0)O) d\lambda dO \, ,
\end{equation*}
where $1_{\mathcal A}$ is the indicator function of the set $\mathcal A$.
It is not hard to see that the mapping
\begin{align*}
K^\ell_t:\mathcal{G}_{k,d}\times \mathcal{G}_{k,d}& \rightarrow \R \\
(P_1,P_2) & \mapsto \int_{\mathscr{H}^\ell_{d}}\langle P_1,M\rangle ^t\langle M,P_2\rangle^t d\nu_{\ell,d}(M)
\end{align*}
is a positive definite kernel on $\mathcal{G}_{k,d}$. Next, we check that the function spaces under consideration are spanned by the shifts of $K^\ell_t$. 
\begin{proposition}\label{prop:to compare with}
If $\ell$ and $t$ are nonnegative integers, then 
\begin{align*}
\Pol_t^\ell(\mathcal{G}_{k,d}) 
&= \spann\{\langle M,\cdot\rangle^t\big|_{\mathcal G_{k,d}} : M\in\mathscr{H}_d,\;\rank(M)=\ell\}\\
&= \spann\{K^\ell_t(P,\cdot)|_{\mathcal{G}_{k,d}} : P\in\mathcal{G}_{k,d} \}.\\
\intertext{If $\ell_1,\ell_2$ and $t_1,t_2$ are nonnegative integers, then}
\Pol^{\ell_1}_{t_1}(\mathcal{G}_{k,d}) \cdot \Pol^{\ell_2}_{t_2}(\mathcal{G}_{k,d}) &\subset \Pol^{\ell_1+\ell_2}_{t_1+t_2}(\mathcal{G}_{k,d}).
\end{align*}

\end{proposition}
\begin{proof}
To verify the first equality, we must check that the left-hand-side is contained in the right-hand-side. 
We first define 
\begin{equation*}
\widetilde{\Pol}^\ell_t(\mathcal{G}_{k,d}):=\spann\{\langle M,\cdot\rangle^t\big|_{\mathcal G_{k,d}} : M\in\mathscr{H}^\ell_d\}. 
\end{equation*}
Since the rank $\ell$ matrices are dense in $\mathscr{H}^\ell_d$, we obtain 
\begin{equation*}
\widetilde{\Pol}^{\ell}_t(\mathcal{G}_{k,d})  = \spann\{\langle M,\cdot\rangle^t\big|_{\mathcal G_{k,d}} : M\in\mathscr{H}_d,\; \rank(M) =\ell  \}.
\end{equation*}
Thus, the first equality holds if we can verify that the spaces $\widetilde{\Pol}^{\ell}_t(\mathcal{G}_{k,d}) $ are an ascending sequence in $t$, i.e., $\widetilde{\Pol}^{\ell}_t(\mathcal{G}_{k,d}) \subset \widetilde{\Pol}^{\ell}_{t+1}(\mathcal{G}_{k,d})$. To do so, we first aim to verify 
\begin{equation}\label{eq:to be proved}
\widetilde{\Pol}^{\ell}_t(\mathcal{G}_{k,d}) = \bigoplus_{t_1+\ldots +t_\ell=t}\widetilde{\Pol}^{1}_{t_1}(\mathcal{G}_{k,d})\cdots \widetilde{\Pol}^{1}_{t_\ell}(\mathcal{G}_{k,d}).
\end{equation}
The spectral decomposition yields that the left-hand-side is contained in the right-hand-side. To verify the reverse set inclusion, we must check that 
\begin{equation}\label{eq:new to prove}
P\mapsto \langle P,x_1x^*_1\rangle^{t_1} \cdots \langle P,x_\ell x^*_\ell\rangle^{t_\ell} \in \widetilde{\Pol}^{\ell}_t(\mathcal{G}_{k,d}),\quad\text{for all } t_1+\ldots+t_\ell=t.
\end{equation}
We now observe that \cite[Lemma 7.1]{Ehler:2014qc} as already used in \eqref{eq:lemma} yields \eqref{eq:new to prove}. Thus, \eqref{eq:to be proved} is satisfied. It was checked in \cite{Bachoc:2010aa} that $\widetilde{\Pol}^{1}_{t}(\mathcal{G}_{k,d})\subset \widetilde{\Pol}^{1}_{t+1}(\mathcal{G}_{k,d})$ holds, so that \eqref{eq:to be proved} implies $\widetilde{\Pol}^{\ell}_t(\mathcal{G}_{k,d}) \subset \widetilde{\Pol}^{\ell}_{t+1}(\mathcal{G}_{k,d})$, which yields the first equality. 

The second part of the proposition follows from the first equality and \eqref{eq:to be proved}.

We now take care of the second equality.  
Since $\mathcal{M}:=\spann\{ \langle \cdot,P_1\rangle^t \langle \cdot,P_2\rangle^t|_{\mathscr{H}^\ell_d}  :  P_1,P_2\in\mathcal{G}_{k,d}\}$ is finite-dimensional, classical arguments let us infer that there are $\{M_j\}_{j=1}^m\subset \mathscr{H}^\ell_d$ and numbers $\{\omega_j\}_{j=1}^m\in\R$ such that, for all $P_1,P_2\in\mathcal{G}_{k,d}$,
\begin{equation}\label{eq:cub new}
\sum_{j=1}^m \omega_j \langle M_j,P_1\rangle^t \langle M_j,P_2\rangle^t = \int_{\mathscr{H}^\ell_d}  \langle M,P_1\rangle^t \langle M,P_2\rangle^t d\nu_{\ell,d}(M),
\end{equation}
cf.~\cite[Theorem 6.1]{Graf:2013zl}. 
By applying \eqref{eq:cub new}, we derive
\begin{align*}
K^\ell_t(P,\cdot)& = \int_{\mathscr{H}^\ell_d} \langle M,P\rangle^t \langle M,\cdot\rangle^t d\nu_{\ell,d}(M)\\
& = \sum_{j=1}^m \omega_j \langle M_j,P\rangle^t \langle M_j,\cdot\rangle^t\in \Pol^\ell_t(\mathcal{G}_{k,d}).
\end{align*} 
Thus, we have verified that $\spann\{K^\ell_t(P,\cdot)|_{\mathcal{G}_{k,d}} : P\in\mathcal{G}_{k,d} \}\subset \Pol^\ell_t(\mathcal{G}_{k,d})$.  To verify the reverse inclusion, we shall check that $\dim(\spann\{K^\ell_t(P,\cdot)|_{\mathcal{G}_{k,d}} : P\in\mathcal{G}_{k,d} \})\geq \dim(\Pol^\ell_t(\mathcal{G}_{k,d}))$. 

We first observe that 
\begin{equation}\label{dim equality}
\dim(\spann\{\langle \cdot,M\rangle^t|_{\mathcal{G}_{k,d}} : M\in\mathscr{H}^\ell_d \}) = \dim(\spann\{\langle P,\cdot\rangle^t|_{\mathscr{H}^\ell_d} : P\in\mathcal{G}_{k,d} \}),
\end{equation}
which is a general principle that holds in much more generality, see \cite[Proof of Lemma 5.5]{Ehler:2014qc} for details. Note that the left-hand-side of \eqref{dim equality} is $\dim(\Pol^\ell_t(\mathcal{G}_{k,d}))$, and we shall denote this number by $r$ here. Then according to \eqref{dim equality}, there are $\{P_j\}_{j=1}^r\subset\mathcal{G}_{k,d}$ such that $\{\langle P_j,\cdot\rangle^t|_{\mathscr{H}^\ell_d}\}_{j=1}^r$ is a basis for $\spann\{\langle P,\cdot\rangle^t|_{\mathscr{H}^\ell_d} : P\in\mathcal{G}_{k,d} \}$. If we can verify that the matrix $K:=\big(K^\ell_t(P_i,P_j) \big)_{i,j=1}^r$ is nonsingular, then $\{K^\ell_t(P_j,\cdot)\}_{j=1}^r$ is linearly independent, which concludes the proof. Indeed, suppose that $\alpha^*K\alpha=0$, then we obtain
\begin{align*}
0 & = \sum_{i,j}\alpha_i\alpha_j K^\ell_t(P_i,P_j)\\
& = \int_{\mathscr{H}^\ell_{d}} \sum_{i,j}\alpha_i\alpha_j\langle P_i,M\rangle ^t\langle M,P_j\rangle^t d\nu_{\ell,d}(M)\\
& =  \int_{\mathscr{H}^\ell_{d}}  \big(  \sum_{i}\alpha_i\langle P_i,M\rangle ^t\big)^2 d\nu_{\ell,d}(M).
\end{align*}
This implies $\sum_{i}\alpha_i\langle P_i,M\rangle ^t=0$, for all $M\in\supp(\nu_{\ell,d})$. Since $\langle P_j,\cdot\rangle^t|_{\mathscr{H}^\ell_d}$ are homogeneous polynomials, the latter also holds for all $M\in\mathscr{H}^\ell_d$. The linear independence of $\{\langle P_j,\cdot\rangle^t|_{\mathscr{H}^\ell_d}\}_{j=1}^r$ implies that we must have $\alpha_1,\ldots,\alpha_r=0$. Thus, $K$ is indeed nonsingular, and this concludes the proof.
\end{proof}
\begin{remark}
The end of the above proof shows that the special form of $\nu_{\ell,d}$ is not important, and any measure with sufficiently large support would work. 
\end{remark}

The kernel $K^\ell_t$ induces an inner product (and hence also a norm $\|\cdot\|_{K^\ell_t}$) on $\Pol^\ell_t(\mathcal{G}_{k,d})$ by 
\begin{equation}\label{eq:express}
\langle f,g\rangle_{{\Pol}^\ell_t}:= \sum_{i,j} \alpha_i\beta_j K^\ell_t(P_i,\tilde{P}_j),
\end{equation}
where $f=\sum_{i}\alpha_i K^\ell_t(P_i,\cdot)$ and $g=\sum_{j}\beta_j K^\ell_t(\tilde{P}_i,\cdot)$. Note that the expression \eqref{eq:express} does not depend on the special choice of $P_i$ and $\tilde{P_j}$.  The induced norm enables us to introduce approximate cubatures:
\begin{definition}\label{def:cubature approx}
We say that $\{(P_j,\omega_j)\}_{j=1}^n$ is an \emph{$\epsilon$-approximate cubature for $ {\Pol}^{\ell}_{t}(\mathcal{G}_{k,d})$} with respect to $K^\ell_t$ if 
 \begin{equation}\label{eq:def cub}
\sup_{f\in {\Pol}^\ell_{t}(\mathcal{G}_{k,d}),\; \|f\|_{K^\ell_t}= 1} |\sum_{j=1}^n\omega_jf(P_j) -\int_{\mathcal{G}_{k,d}} f(P)d\sigma_{k,d}(P)|\leq \epsilon.
\end{equation}
\end{definition}

Apparently, an $\epsilon$-approximate cubature for $ {\Pol}^{\ell}_{t}(\mathcal{G}_{k,d})$ yields
\begin{equation*}
\big|\sum_{j=1}^n\omega_jf(P_j) -\int_{\mathcal{G}_{k,d}} f(P)d\sigma_{k,d}(P)\big|\leq \epsilon\|f\|_{K^\ell_t},\quad\text{for all }f\in {\Pol}^\ell_{t}(\mathcal{G}_{k,d}).
\end{equation*}

In order to numerically find $\epsilon$-approximate cubatures, we consider the modified fusion frame potential
\begin{equation}\label{eq:modi ffp}
\sum_{i,j} \omega_i\omega_jK^\ell_t(P_i,P_j).
\end{equation}
By following the lines in \cite{Ehler:2014zl} for the standard fusion frame potential, see also \cite{Bachoc:2010aa}, we derive that
\begin{equation}\label{eq:lower bound 2}
c^\ell_t:= \int_{\mathcal{G}_{k,d}} \int_{\mathcal{G}_{k,d}}  K^\ell_t(P,Q) d\sigma_{k,d}(P)d\sigma_{k,d}(Q)
\end{equation}
is a lower bound on \eqref{eq:modi ffp},
and the gap 
\begin{equation*}
\sum_{i,j} \omega_i\omega_jK^\ell_t(P_i,P_j)-c^\ell_t \geq 0
\end{equation*}
is exactly the squared cubature error, i.e.,
\begin{equation*}
\sum_{i,j} \omega_i\omega_jK^\ell_t(P_i,P_j)-c^\ell_t = \sup_{f\in {\Pol}^\ell_{t}(\mathcal{G}_{k,d}),\; \|f\|_{K^\ell_t}= 1} |\sum_{j=1}^n\omega_jf(P_j) -\int_{\mathcal{G}_{k,d}} f(P)d\sigma_{k,d}(P)|^2.
\end{equation*}
Indeed, if \eqref{eq:modi ffp} can be minimized numerically, then a proper cubature or at least an $\epsilon$-approximate cubature can be obtained, where $\epsilon$ relates to machine precision provided there exists a corresponding cubature for this choice of $n$. However, numerical evaluation of the kernel $K^\ell_t$ may be difficult in practice. In the subsequent section, we shall circumvent such difficulties by considering cubatures for larger spaces that enable us to work with a simpler kernel.

\section{Construction of cubatures for $\Pol_{t}(\mathcal{G}_{k,d})$}\label{sec:Polt}
\subsection{Cubatures from optimization procedures}
This section is dedicated to derive cubatures from a numerical scheme that is indeed easy to implement. We define polynomials of degree at most $t$ on $\mathcal{G}_{k,d}$ by
\begin{equation}\label{eq:general pol space}
\Pol_t(\mathcal{G}_{k,d}) :=\{\text{polynomials of degree at most $t$ on $\mathscr{H}_d$ restricted to $\mathcal{G}_{k,d}$}\}.
\end{equation}
Note that $\Pol_t(\mathcal{G}_{k,d})$ satisfies the product property that is usually associated with polynomial spaces, i.e., 
\begin{equation*}
\spann\big(\Pol_{t_1}(\mathcal{G}_{k,d})\cdot \Pol_{t_2}(\mathcal{G}_{k,d}) \big) = \Pol_{t_1+t_2}(\mathcal{G}_{k,d}),
\end{equation*}
see, for instance, \cite{Ehler:2014zl}. It is known that these spaces can be rewritten as
\begin{equation*}
\Pol_t(\mathcal{G}_{k,d}) =\spann\{\langle M,\cdot\rangle^t\big|_{\mathcal G_{k,d}} : M\in\mathscr{H}_d\},
\end{equation*}
see, for instance, \cite{Ehler:2014zl,Bachoc:2010aa}. 
Obviously, $\Pol^\ell_{t}(\mathcal{G}_{k,d})$ is contained in $\Pol_{t}(\mathcal{G}_{k,d})$. In the following proposition, we explore when equality holds: 
\begin{proposition}\label{prop:3.1}
For $0\leq \ell<d$ and $0\leq t$, we have
\begin{equation*}
\Pol^{\ell}_t(\mathcal{G}_{k,d})  = \Pol_t(\mathcal{G}_{k,d}),\quad\text{for }\ell\geq \min\{k,t\}.
\end{equation*}
\end{proposition}
\begin{proof}
For $\ell\geq k$, the equality is standard cf.~\cite{Ehler:2014zl}. We derive from \eqref{eq:to be proved} that 
\begin{equation*}
\Pol^{t}_t(\mathcal{G}_{k,d}) = \spann \big( \Pol^{1}_{1}(\mathcal{G}_{k,d})\cdots \Pol^{1}_{1}(\mathcal{G}_{k,d})\big),
\end{equation*}
where the product has $t$ terms, 
holds, and the findings in \cite{Bachoc:2010aa} yield that the right-hand-side equals $\Pol_t(\mathcal{G}_{k,d})$. Therefore, $\ell\geq t$ also yields $\Pol^{\ell}_t(\mathcal{G}_{k,d})=\Pol_t(\mathcal{G}_{k,d})$. 
\end{proof}
Note that $\{(P_j,\omega_j)\}_{j=1}^n$ being a cubature for $\Pol^{2}_{3}(\mathcal{G}_{k,d})$ as used in Corollary \ref{main rec theorem} already implies that it is also a cubature for $\Pol_{2}(\mathcal{G}_{k,d})=\Pol^2_{2}(\mathcal{G}_{k,d})$ and for $\Pol_{1}(\mathcal{G}_{k,d})=\Pol^2_{1}(\mathcal{G}_{k,d})$. It should also be mentioned that the space $\Pol^2_{t}(\mathcal{G}_{1,d})$ in Corollary \ref{cor:zweiter Versuch} is the same as $\Pol_{t}(\mathcal{G}_{1,d})$. Hence, for $k=1$, we were dealing with the space \eqref{eq:general pol space} all along.


A computational approach for cubatures for $\Pol_{t}(\mathcal{G}_{k,d})$ is discussed in \cite{Ehler:2014zl}. Since $\Pol^{2}_{t}(\mathcal{G}_{k,d})$ is a subset, this approach yields also cubatures for $\Pol^{2}_{t}(\mathcal{G}_{k,d})$. By refining some ideas in \cite{Ehler:2014zl}, we shall introduce $\epsilon$-approximate cubatures for the kernel
\begin{align*}
K_t:\mathcal{G}_{k,d}\times \mathcal{G}_{k,d}& \rightarrow \R \\
(P_1,P_2) & \mapsto \langle P_1,P_2\rangle^t.
\end{align*}
Indeed, $K_t$ is a positive definite kernel on $\mathcal{G}_{k,d}$ and its shifts generate $\Pol_{t}(\mathcal{G}_{k,d})$, i.e.,  
\begin{equation*}
\Pol_t(\mathcal{G}_{k,d}) = \spann\{K_t(P,\cdot) : P\in\mathcal{G}_{k,d} \}.
\end{equation*}
The kernel $K_t$ induces  an inner product on $\Pol_t(\mathcal{G}_{k,d})$ analogously to \eqref{eq:express} and, in turn, also a norm $\|\cdot\|_{K_t}$. 
\begin{definition}
We say that $\{(P_j,\omega_j)\}_{j=1}^n$ is an $\epsilon$-approximate cubature for $ {\Pol}_{t}(\mathcal{G}_{k,d})$ with respect to $K_t$ if 
 \begin{equation*}
\sup_{f\in {\Pol}_{t}(\mathcal{G}_{k,d}),\; \|f\|_{K_t}= 1} |\sum_{j=1}^n\omega_jf(P_j) -\int_{\mathcal{G}_{k,d}} f(P)d\sigma_{k,d}(P)|\leq \epsilon.
\end{equation*}
\end{definition}


In the following we shall describe that $\epsilon$-approximate cubatures for $ {\Pol}_{t}(\mathcal{G}_{k,d})$ can be computed by numerical schemes as at the end of Section \ref{sec:4.2}. Indeed, the potential 
\begin{equation}\label{eq:ffp 1}
\sum_{i,j} \omega_i\omega_j K_t(P_i,P_j)
\end{equation}
can be bounded from below by 
\begin{equation}\label{eq:lambda}
\lambda_t:=\int_{\mathcal{G}_{k,d}}\int_{\mathcal{G}_{k,d}}K_t(P,P')d\sigma_{k,d}(P)d\sigma_{k,d}(P'),
\end{equation}
so that 
\begin{equation*}
\sum_{i,j} \omega_i\omega_jK_t(P_i,P_j)-\lambda_t\geq 0.
\end{equation*}
As in the previous section, this gap is exactly the squared cubature error, i.e., 
\begin{equation*}
\sum_{i,j}\omega_i\omega_jK_t(P_i,P_j)-\lambda_t = \sup_{f\in {\Pol}_{t}(\mathcal{G}_{k,d}),\; \|f\|_{K_t}= 1} |\sum_{j=1}^n\omega_jf(P_j) -\int_{\mathcal{G}_{k,d}} f(P)d\sigma_{k,d}(P)|^2.
\end{equation*}
It is remarkable that \eqref{eq:lambda} can be computed exactly by analytical tools, so that the outcome of numerical optimization schemes minimizing \eqref{eq:ffp 1} can be compared with $\lambda_t$, see \cite{Ehler:2014zl} for further details and examples of successful minimization outcomes. Indeed, the easier structure of the kernel $K_t$ generating $\Pol_t(\mathcal{G}_{k,d})$ make this approach more amenable to numerical optimization than the setting of $\Pol^\ell_t(\mathcal{G}_{k,d})$ presented in the previous section. 


\subsection{Approximate cubatures from randomized projections}
We now examine to what extent a random choice of projections gives an approximate cubature. Let us call a (Borel)-probability measure $\nu_{k,d}$ on $\mathcal{G}_{k,d}$ a \emph{probabilistic cubature} for $\Pol_t(\mathcal{G}_{k,d})$ if
\begin{equation}\label{prob def}
\int_{\mathcal{G}_{k,d}} f(P)d\nu_{k,d}(P) = \int_{\mathcal{G}_{k,d}} f(P)d\sigma_{k,d}(P),\quad\text{for all } f\in \Pol_t(\mathcal{G}_{k,d}).
\end{equation}
Note that any cubature for $\Pol_{t}(\mathcal{G}_{k,d})$ can be considered as a finitely supported probabilistic cubature, provided that the weights are nonnegative. Another example, of course, is $\sigma_{k,d}$ itself. 

In the remainder of this section, we let each $\omega_j = \frac 1 n$ and choose each $P_j$ according to a probabilistic cubature $\nu_{k,d}$. In that case, $\{P_j\}_{j=1}^n$ is a collection of random matrices and the expected value of the gap, that is, the squared cubature error, can be computed explicitly. Denoting the expectation with respect
to the random choice of $\{P_j\}_{j=1}^n$ by $\mathbb E_P$, and using that 
$\mathbb E_P K_t(P_i,P_i) = k^t
$ 
and $\mathbb E_P K_t(P_i,P_j) = \lambda_t$ if $i \ne j$, we get
\begin{align*}
 \mathbb E_P \big[ \frac{1}{n^2} \sum_{i,j=1}^n K_t(P_i, P_j) - \lambda_t  \big]
 & = \frac{k^t}{n} + \frac{n(n-1)}{n^2} \lambda_t - \lambda_t = \frac{1}{n}(k^t - \lambda_t).
\end{align*}
Thus, letting $n$ grow faster than  $k_t$ ensures that the expected value of the gap 
becomes arbitrarily small. In the following theorem, we show that this expected behavior happens with overwhelming probability.
\begin{theorem}
If $\{P_j\}_{j=1}^n$ are chosen independently identically distributed with respect to a probabilistic cubature $\nu_{k,d}$ for  $\Pol_t(\mathcal{G}_{k,d})$ and $\tau>0$, then 
\begin{equation*}
   \mathbb P\big(\frac{1}{n^2} \sum_{i,j=1}^n K_t(P_i, P_j) -\lambda_t - \frac{1}{n}(k^t - \lambda_t)\ge \frac{\tau^2 k^t}{n}\big)
    \le  4 e^{-\Psi_\tau(n)} r_\tau(n),
\end{equation*}
where 
\begin{equation*}
   \Psi_\tau(n) = \frac{\tau^2/2}{ (1 -  \lambda_t/k^t)+\frac{\tau}{3\sqrt{n}}},\qquad 
   r_\tau(n) = 1 + \frac{6}{ n \tau^2 \ln^2(1+\frac{ \tau}{\sqrt{n} (1 -  \lambda_t/k^t)})}.
\end{equation*}
\end{theorem}
\begin{proof}
First, we note that $\langle P_i, P_j \rangle^t = \langle P_i^{\otimes t}, P_j^{\otimes t}\rangle$,
where the Hilbert-Schmidt inner product on the right-hand side is on the Hilbert space $(\mathbb R^{d\times d})^{\otimes t} \simeq \mathbb R^{d^2t}$.
Thus, $\sum_{i,j} K_t(P_i, P_j) = \|\sum_{j} P_j^{\otimes t}\|_{HS}^2$ holds. 

We define the averaged tensor power $\Lambda_t = \mathbb E_{P} P_1^{\otimes t}$. 
 For $P_0\in\mathcal{G}_{k,d}$, we can compute
\begin{align*}
\langle P_0^{\otimes t},\Lambda_t\rangle & = \langle P_0^{\otimes t},\int_{\mathcal{G}_{k,d}}\hspace{-.2cm} P^{\otimes t}d\nu_{k,d}(P)\rangle
 =\int_{\mathcal{G}_{k,d}}\hspace{-.2cm}  \langle P_0,P\rangle^t d\nu_{k,d}(P).
 \intertext{Since $\nu_{k,d}$ is a probabilistic cubature and $\langle P_0,\cdot\rangle^s\in \Pol_t(\mathcal{G}_{k,d})$, we obtain} 
 \langle P_0^{\otimes t},\Lambda_t\rangle & = \int_{\mathcal{G}_{k,d}}\hspace{-.2cm}  \langle P_0,P\rangle^t d\sigma_{k,d}(P).
 \intertext{Let $U\in\mathcal{O}_d$ be such that $U^*D_kU=P_0$, where $D_k$ denotes the diagonal matrix with $k$ ones and zeros elsewhere. The commutativity of the trace and the orthogonal invariance of $\sigma_{k,d}$ yield}
 \langle P_0^{\otimes t},\Lambda_t\rangle 
&=\int_{\mathcal{G}_{k,d}}\hspace{-.2cm}   \langle U^*D_kU,P\rangle^t d\sigma_{k,d}(P)\\
 &=\int_{\mathcal{G}_{k,d}}\hspace{-.2cm}   \langle D_k,UPU^*\rangle^t d\sigma_{k,d}(P)\\
 &= \int_{\mathcal{G}_{k,d}}\hspace{-.2cm}   \langle D_k,P\rangle^t d\sigma_{k,d}(P).
 \intertext{By applying the probabilistic cubature property once more, we derive}
 \langle P_0^{\otimes t},\Lambda_t\rangle & = \int_{\mathcal{G}_{k,d}}\hspace{-.2cm}   \langle D_k,P\rangle^t d\nu_{k,d}(P)= \langle D_k^{\otimes t},\Lambda_t\rangle.
 \end{align*}
Thus, the term $\langle P_0^{\otimes t},\Lambda_t\rangle$ does not depend on the particular choice of $P_0\in\mathcal{G}_{k,d}$. 
Averaging over all $P_0$  with respect to $\sigma_{k,d}$ then implies that for each $i$,
\begin{equation}\label{eq:une}
\langle P_i^{\otimes t},\Lambda_t\rangle= \int_{\mathcal{G}_{k,d}}\langle P_0^{\otimes t},\Lambda_t\rangle
d\sigma_{k,d}(P_0)= \|\Lambda_t\|_{HS}^2.
\end{equation}
Similarly to the above computations, the probabilistic cubature property also yields
\begin{equation}\label{eq:deux}
\|\Lambda_t\|_{HS}^2=\lambda_t.
\end{equation} 
By applying \eqref{eq:une} and \eqref{eq:deux}, taking $Y_j = (P_j^{\otimes t} - \Lambda_t)/k^{t/2}$ then gives  
\begin{align*}
\|Y_j\|_{HS}^2  = 1 -\lambda_t/k^t.
\end{align*}
Hence, $\|Y_j\|_{HS} \leq 1$ and $\mathbb{E}_PY_j = 0$, so that Minsker's vector-valued Bernstein inequality \cite[Corollary 5.1]{Minsker:2011} provides, for all $\tau >0$, 
\begin{equation*}
   \mathbb P\big( \|\frac{1}{n} \sum_{j=1}^n Y_j \|^2_{HS} > \tau^2/n\big) \le 4 e^{-\Psi_{\tau}(n)} r_{\tau}(n),
\end{equation*}
where $\Psi_\tau$ and $r_\tau$ are as stated. 
To finish the proof, we observe that 
\begin{equation*}
k^t\|\frac{1}{n} \sum_{j=1}^n Y_j \|^2_{HS} = \frac{1}{n^2} \sum_{i,j=1}^n K_t(P_i, P_j) -\lambda_t - \frac{1}{n}(k^t - \lambda_t). \qedhere
\end{equation*}
\end{proof}

When $n$ tends to infinity, then $r_\tau(n) \to 1 + \frac{6(1 -\lambda_t/k^t )^2}{ \tau^4}$ 
and $\Psi_\tau(n) \to \frac{\tau^2/2}{1 - \lambda_t/k^t}$. 
Thus, for large $n$, the distribution of the gap concentrates near zero at the same rate as the expected value.
Since the gap is the square of the maximal cubature error, we conclude a probabilistic construction of approximate
cubatures.
 \begin{corollary}
 If $\{P_j\}_{j=1}^n$ are chosen independently from a probabilistic cubature for $\Pol_t(\mathcal{G}_{k,d})$ and $\tau>0$, then a $\sqrt{\frac{(1+\tau^2)k^t-\lambda_t}{ n}}$-approximate
cubature for $ {\Pol}_{t}(\mathcal{G}_{k,d})$ with respect to $K_t$ is obtained with probability bounded below by
$
1 - 4 e^{-\Psi_\tau(n)} r_\tau(n).
$
\end{corollary}
For related results on random matrices, we refer to \cite{Bodmann:2013tg,Levina:2012lq,Vershynin:2011ai,Vershynin:2010aa}.

\section{Error propagation for $\epsilon$-approximate cubatures}\label{sec:error}
The numerical optimization approach in general can provide cubatures up to machine precision only. Therefore, we are dealing with $\epsilon$-approximate cubatures and this is also what we obtain from the random constructions. In these cases, the moment reconstruction formulas in Corollary \ref{main rec theorem} hold up to some error term:

\begin{theorem}
Let $X\in \mathbb{S}^{d-1}$ be a random vector and $\{(P_j,\omega_j)\}_{j=1}^n$ be an $\epsilon$-approximate cubature for $\Pol_{t}(\mathcal{G}_{k,d})$ with respect to $K_t$. Then \eqref{eq:fundament} in Corollary \ref{cor:3} holds up to a constant $c_\alpha$ times $\epsilon$, i.e., for $\alpha\in\N^d$, $|\alpha|=t$, 
\begin{equation}\label{eq:eps error}
\big| \mathbb{E} X^\alpha -  \sum_{|\beta|\leq t}  a^\alpha_\beta \sum_{j=1}^n \omega_j\mathbb{E} (P_jX)^\beta  \big| \leq \epsilon c_\alpha.
\end{equation}
If $k=1$ and $X$ is random vector in $\R^d$, then \eqref{eq:fundament 2} in Corollary \ref{cor:zweiter Versuch} holds up to a constant times $\epsilon \mathbb{E}\|X\|^t$. 
\end{theorem}
The above theorem verifies that the cubature error propagates in a linear fashion when it comes to the moment reconstruction formulas. It should be mentioned though that the constant $c_\alpha$ depends on $k$ and $d$.
\begin{proof}
According to Theorem \ref{th:fundamental}, we derive, for $x\in \mathbb{S}^{d-1}$
\begin{align*}
x^\alpha &= \sum_{s=1}^t \sum_{i=1}^m f_{s,i}\mu^s_{k,d}(E_{x,y_i})\\
& = \sum_{s=1}^t \sum_{i=1}^m f_{s,i}\int_{\mathcal{G}_{k,d}} \langle P,E_{x,y_i}\rangle^s d\sigma_{k,d}(P)
\end{align*}
Since the function $F^\alpha_x = \sum_{s=1}^t \sum_{i=1}^m f^\alpha_{s,i} \langle \cdot,E_{x,y^\alpha_{s,i}}\rangle^s$ is an element in $\Pol_t(\mathcal{G}_{k,d})$, the cubature property yields
\begin{equation*}
\big| x^\alpha - \sum_{s=1}^t \sum_{i=1}^m f^\alpha_{s,i}\sum_{j=1}^n \omega_j \langle P_j,E_{x,y^\alpha_{s,i}}\rangle^s \big| \leq \epsilon \|F_x\|_{K_t}.
\end{equation*}
The coefficients $a^\alpha_\beta$ in Corollary \ref{cor:3} are used with $c_\alpha:=\sup_{x\in \mathbb{S}^{d-1}} \|F^\alpha_x\|_{K_t}$ to derive \eqref{eq:eps error}. 

The second part of the theorem can be verified in an analogous fashion, so we omit the details. 
\end{proof}

\section{Concluding remarks}
Our results appear to match reasonable characteristics in distributed sensing. We require a rather large set of sensors (projectors) and we assume that the high-dimensional signal is modeled by means of a probability distribution. The sensors are deterministic and can even be given by the experimental setup as long as we are able to find weights such that projectors and weights altogether form a cubature. Each sensor must reconstruct the first few moments of the projection marginal distribution, which may allow in practice for fewer data samples than for estimating the marginal distribution itself. 
In the end, the first few moments of the high-dimensional random signal can be computed with low costs by a closed formula.

As far as we know, the present paper is a first attempt to address this type of moment recovery problem with tools from harmonic analysis. Further investigations are necessary to combine those ideas with proper statistical estimation techniques, in which the low-dimensional moments are estimated from acquired data. This is intended to be addressed in forthcoming work.

\section*{Acknowledgments}
M.~E.~ and M.~G.~have been funded by the Vienna Science and Technology Fund (WWTF) through project VRG12-009.
B.~G.~B. has been supported by NSF grant DMS-1412524. M.~E.~would like to thank Christian Krattenthaler and Herwig Hauser for providing references on irreducibility for Section \ref{sec:random negative weights}.


\providecommand{\bysame}{\leavevmode\hbox to3em{\hrulefill}\thinspace}
\providecommand{\MR}{\relax\ifhmode\unskip\space\fi MR }
\providecommand{\MRhref}[2]{%
  \href{http://www.ams.org/mathscinet-getitem?mr=#1}{#2}
}
\providecommand{\href}[2]{#2}

%
%

\end{document}